\documentclass[12pt,a4paper]{amsart}
\usepackage{etex} 
\usepackage[latin1]{inputenc}
\usepackage[T1]{fontenc}
\usepackage[right=1cm,left=2cm,top=3cm,bottom=4cm]{geometry}
\usepackage{ifthen}
\usepackage{enumerate}
\usepackage{amsmath}
\usepackage{amsthm}
\usepackage{amsfonts}
\usepackage{amssymb}
\usepackage{latexsym}
\usepackage{ae}
\usepackage[normalem]{ulem}

\title{Short interval results for a class of arithmetic functions}
\author{Olivier Bordellès}
\address{2 allée de la combe \\ 43000 Aiguilhe \\ France}
\email{borde43@wanadoo.fr}
\date{}
\dedicatory{}

\newcommand{\Z}{\mathbb {Z}}

\newcommand{\R}{\mathbb {R}}
\newcommand{\C}{\mathbb {C}}

\DeclareMathOperator{\id}{Id}

\theoremstyle{plain}
\newtheorem{theorem}{Theorem}
\newtheorem{corollary}[theorem]{Corollary}
\newtheorem{lemma}[theorem]{Lemma}
\newtheorem{proposition}[theorem]{Proposition}

\theoremstyle{definition}

\theoremstyle{remark}
\newtheorem{remark}[theorem]{Remark}

\makeatletter
\@namedef{subjclassname@2010}{%
  \textup{2010} Mathematics Subject Classification}
\makeatother

\begin{document}

\begin{abstract}
Using estimates on Hooley's $\Delta$-function and a short interval version of the celebrated Dirichlet hyperbola principle, we derive an asymptotic formula for a class of arithmetic functions over short segments. Numerous examples are also given.
\end{abstract}


\maketitle

\thispagestyle{myheadings}
\font\rms=cmr8 
\font\its=cmti8 
\font\bfs=cmbx8

\section{Introduction and result}
\label{s1}

\noindent
Studying the behaviour of arithmetic functions in short intervals is a quite long-standing problem in number theory. By 'short intervals' we mean the study of sums of the shape
$$\sum_{x < n \leqslant x+y} F(n)$$
where $y=o(x)$ as $x \longrightarrow \infty$. \\ 

\noindent
One of the first results to be published dealing with this problem was Ramachandra's benchmarking paper \cite{ram} revisited by K\'{a}tai \& Subbarao \cite{kat}. More recently, Cui \& Wu \cite{cui} derived a short interval version of the Selberg-Delange method, developed by Selberg and Delange between 1954 and 1971 in order to provide a quite general theorem giving the right order of magnitude of the usual arithmetic functions. \\

\noindent
For multiplicative functions $F$ such that $F(p)$ is close to $1$ for every prime $p$, another method was developed in \cite{bor} using profound theorems from Filaseta-Trifonov and Huxley-Sargos results on integer points near certain smooth curves. This leads to very precise estimates of a large class of multiplicative functions. \\

\noindent
In this work, we derive asymptotic results for short sums of arithmetic functions $F$ such that there exist $s \in \Z_{\geqslant 0}$, $m \in \Z_{\geqslant 1}$ and real numbers $\kappa \in \left[ 0,1 \right)$, $\beta \geqslant 0$, $\delta \geqslant 0$ and $A >0$ such that
\begin{eqnarray}
   & & \sum_{n \leqslant z} \left( F \star \mu \right) (n) = z \sum_{j=0}^s a_j (\log z)^j + O \left( z^\kappa (\log ez)^\beta \right) \quad \left(z \geqslant 1, \ a_j \in \C, \ a_s \neq 0 \right ) \label{e1} \\
   & & \notag \\
   & & \left | \left( F \star \mu \right) (n) \right | \leqslant A \, \tau_m(n) \, (\log e n)^\delta \quad \left( n \in \Z_{\geqslant 1} \right) \label{e2}
\end{eqnarray}
where, as usual, $\mu$ is the M\"{o}bius function, $\tau_m$ is the $m$th Dirichlet-Piltz divisor function and $f \star g$ is the Dirichlet convolution product of the arithmetic functions $f$ and $g$ defined by
$$(f \star g) (n) = \sum_{d \mid n} f(d) g(n/d).$$

\noindent
Our main result may be stated as follows.

\begin{theorem}
\label{t1}
Let $e^e < y \leqslant x$ be real numbers and $F$ be an arithmetic function satisfying \eqref{e1} and \eqref{e2}. Then, for any $\varepsilon \in \left(  0,\frac{1}{2} \right]$ and $x^{1/2} e^{-\frac{1}{2}(\log x)^{1/4}} \leqslant y \leqslant x e^{-(\log x)^{1/4}}$  
\begin{eqnarray*}
   \sum_{x < n \leqslant x+y} F(n) &=& \frac{y a_s}{s+1} (\log x)^{s+1} + O \Bigl (  xy^{\kappa-1} \, e^{(\kappa -1) (\log x)^{1/4}} \, \left( \log x \right)^\beta  \\
   & &   {} + y (\log x)^{\max\left (s,\delta + m - \frac{1}{2} + \epsilon_{m+1} (x) \right )} + x^{ \varepsilon} \Bigr ).
\end{eqnarray*}
The implied constant depends only on $A,\varepsilon,m,s,a_0,\dotsc,a_s$ and on the implied constant arising in \eqref{e1}. For any integer $r \geqslant 2$, the function $\epsilon _r (x) = o(1)$ as $x \longrightarrow \infty$ is given in \eqref{e3} below.
\end{theorem}

\noindent
It should be mentioned that if the error term in \eqref{e1} is of the form $O \left( z^{\kappa + \epsilon} \right)$ instead of $O \left( z^\kappa (\log ez)^\beta \right)$, then the proof shows that the term $xy^{\kappa-1} \, e^{(\kappa -1) (\log x)^{1/4}} \, \left( \log x \right)^\beta$ has to be replaced by
$$xy^{\kappa-1+ \varepsilon} \, e^{(\kappa -1 + \varepsilon) (\log x)^{1/4}}$$
the other terms remaining unchanged.

\section{Notation}
\label{s2}

\noindent
In what follows, $e^e < y \leqslant x$ are large real numbers, $N \geqslant 1$ is a large integer and $\varepsilon >0$ is an arbitrary small real number which does not need to be the same at each  occurence. \\

\noindent
For any $t \in \R$, $\left \lfloor t \right \rfloor$ is the integer part of $t$, $\psi \left( t \right) = t - \left \lfloor t \right \rfloor - \frac{1}{2}$ is the first Bernoulli function, $e(t) = e^{2 \pi i t}$ and $\| t \|$ is the distance from $t \in \R$ to its nearest integer. \\

\noindent
For any integer $r \geqslant 2$ and any real number $x > e^e$, we set
\begin{equation}
   \epsilon _r (x) = \sqrt{\frac{r \log \log \log x}{\log \log x}} \left( r-1+\frac{30}{\log \log \log x}\right). \label{e3}
\end{equation}

\noindent
Besides the arithmetic functions already listed above, let $\Delta_r$ be the $r$th Hooley's divisor function defined by
$$\Delta_r (n) = \underset{u_1,\dotsc,u_{r-1} \in \R}{\max} \sum_{\substack{d_1 d_2 \dotsb d_{r-1} \mid n \\ e^{u_i} < d_i \leqslant e^{u_i + 1}}} 1$$
where $r \geqslant 2$ is a fixed integer. \\

\noindent
Finally, let $F$ be any arithmetic function and let $f = F \star \mu$ be the Eratosthenes transform of $F$. We set
\begin{equation}
   \Sigma_F(N,x,y) := \sum_{N < n \leqslant 2N} f(n) \left( \psi \left( \frac{x+y}{n} \right) - \psi \left( \frac{x}{n} \right) \right) \quad \left( y < N \leqslant x \right). \label{e4}
\end{equation}

\section{First technical tools}
\label{s3}

\begin{lemma}
\label{le1} 
Let $F$ be any arithmetic function satisfying \eqref{e1}, and $f = F \star \mu$ be the Eratosthenes transform of $F$. Then, for any real numbers $1 \leqslant t \leqslant z$
$$\sum_{z < n \leqslant z+t} f(n) = t a_s (\log z)^s + O \left\lbrace  t (\log ez)^{s-1} + z^\kappa (\log ez)^\beta \right\rbrace.$$
\end{lemma}

\begin{proof}
Follows from \eqref{e1} and usual arguments. We leave the details to the reader.
\end{proof}

\begin{lemma}
\label{le2} 
Let $F$ be any arithmetic function and $f = F \star \mu$ be the Eratosthenes transform of $F$. Then, for any integers $N \in \left( y,x \right]$ and $H \geqslant 1$
\begin{eqnarray*}
   \Sigma_F(N,x,y) &=& - \int_x^{x+y} \sum_{0 < \left| h \right| \leqslant H} \sum_{N < n \leqslant 2N} \frac{f(n)}{n} \, e \left( \frac{ht}{n} \right) \, \mathrm{d}t \\
   & & {} + O \left( \max_{x \leqslant z \leqslant x+y} \sum_{N < n \leqslant 2N} \left | f(n) \right | \min \left( 1,\frac{1}{H \left\| z/n\right\| }\right) \right)
\end{eqnarray*}
where $\Sigma_F(N,x,y)$ is given in \eqref{e4}.
\end{lemma}

\begin{proof}
We use the representation 
$$\psi(t) = -\sum_{0 < \left| h \right| \leqslant H} \frac{e \left( ht \right)}{2\pi ih} + O \left\{ \min \left( 1,\frac {1}{H \left\| t\right\| }\right) \right\} \quad \left( H \in \Z_{\geqslant 1} \right) $$
to get 
\begin{eqnarray*}
   \sum_{N < n \leqslant 2N} f(n) \psi \left( \frac{x}{n}\right) &=& - \sum_{0 < \left| h \right| \leqslant H}\frac{1}{2\pi ih}\sum_{N < n \leqslant 2N} f(n) e \left( \frac{hx}{n} \right) \notag \\
   & & {} + O \left( \sum_{N < n \leqslant 2N} \left | f(n) \right | \min \left( 1,\frac{1}{H \left\| x/n\right\| }\right) \right)
\end{eqnarray*}
and the result follows using the identity 
$$e \left( a (x+y) \right) -e \left( ax \right) = 2 \pi i a \int_x^{x+y} e \left( at \right) \textrm{d} t.$$
The proof is complete.
\end{proof}

\section{Dirichlet's hyperbola principle in short intervals}
\label{s5}

\noindent
The Dirichlet hyperbola principle was first developed by Dirichlet circa 1840 which enabled him to prove the first non-trivial asymptotic result for the so-called Dirichlet divisor problem. It was later generalized in order to deal with long sums of convolution products of the shape $f \star g$ and usually provides good estimates when $f(n)$ or $g(n)$ is $\ll n^\varepsilon$. \\

\noindent
The next result is a short sum version of this very useful principle.

\begin{lemma}
\label{le3} 
Let $x,y,T \in \R$ such that $1 \leqslant \max \left (y,\frac{x}{y} \right) \leqslant T \leqslant x$. Then, for any arithmetic functions $f$ and $g$
\begin{eqnarray*}
   \sum_{x < n \leqslant x+y} \left( f \star g \right) (n) &=& \sum_{d \leqslant T} f(d) \sum_{\frac{x}{d} < k \leqslant \frac{x+y}{d}} g(k) + \sum_{k \leqslant x/T} g(k) \sum_{\frac{x}{k} < d \leqslant \frac{x+y}{k}} f(d) \\
   & & {} + O \left ( \max_{k \leqslant 2x/T} |g(k)|  \sum_{T < d \leqslant T(1+y/x)} \left| f(d) \right | \right).
\end{eqnarray*}
\end{lemma}

\begin{proof}
Let $1 \leqslant \max \left (y,\frac{x}{y} \right) \leqslant T \leqslant x$. Then
\begin{eqnarray*}
   \sum_{x < n \leqslant x+y} \left( f \star g \right)(n) &=&\sum_{d \leqslant x+y} f(d) \sum_{\frac xd<k\leqslant \frac{x+y}d} g\left( k\right) \\
   &=&\sum_{d \leqslant T}f(d) \sum_{\frac{x}{d} < k \leqslant \frac{x+y}{d}} g(k) +\sum_{T < d \leqslant x+y} f(d) \sum_{\frac{x}{d} < k \leqslant \frac{x+y}{d}} g(k) \\
   &=&\sum_{d \leqslant T} f(d) \sum_{\frac{x}{d} < k \leqslant \frac{x+y}{d}} g(k) + S_1\left( T\right)
\end{eqnarray*}
where
\begin{eqnarray*}
   S_1(T) &=& \sum_{T < d \leqslant x+y }f(d) \sum_{k \leqslant \frac{x+y}{d}} g(k) - \sum_{T < d \leqslant x+y} f(d) \sum_{k \leqslant \frac{x}{d}} g(k) \\
   &=& \sum_{k\leqslant \frac{x+y}{T}} g(k) \sum_{T < d \leqslant \frac{x+y}{k}} f(d) -\sum_{k \leqslant \frac{x}{T}} g(k) \sum_{T <d \leqslant \frac{x}{k}} f(d) \\
   &=&\sum_{k\leqslant \frac xT}g\left( k\right) \sum_{T<d\leqslant \frac{x+y}k}f\left( d\right) +\sum_{\frac{x}{T} < k \leqslant \frac{x+y}{T}} g(k) \sum_{T < d \leqslant \frac{x+y}{k}} f(d) - \sum_{k \leqslant \frac{x}{T}} g(k) \sum_{T < d \leqslant \frac{x}{k}} f(d) \\
   &=& \sum_{k \leqslant x/T} g(k) \sum_{\frac{x}{k} < d \leqslant \frac{x+y}{k}} f(d) +\sum_{\frac{x}{T} < k \leqslant \frac{x+y}{T}} g(k) \sum_{T < d \leqslant \frac{x+y}{k}} f(d) \\
   &=& \sum_{k \leqslant x/T} g(k) \sum_{\frac{x}{k} < d \leqslant \frac{x+y}{k}} f(d) + S_2(T)
\end{eqnarray*}
say. Since $y \leqslant T \leqslant x$, the interval $\left( \frac{x}{T} , \frac{x+y}{T} \right] $ contains at most an integer, namely $\left \lfloor \frac{x+y}{T} \right \rfloor$, so that 
$$S_2(T) = g \left( \left \lfloor \frac{x+y}{T} \right \rfloor \right) \left( \left \lfloor \frac{x+y}{T} \right \rfloor - \left \lfloor \frac{x}{T} \right \rfloor \right) \sum_{T < d \leqslant \frac{x+y}{\lfloor (x+y)/T \rfloor}} f(d).$$
Note that $S_2(T) = 0$ in the following cases : \\
\begin{itemize}
   \item[-] either the interval $\left( \frac{x}{T} , \frac{x+y}{T} \right]$ does not contain any integer; \\
   \item[-] or $\frac{x+y}{T} \in \Z$, for then $\frac{x+y}{\lfloor (x+y)/T \rfloor}  = T$ and the inner sum of $S_2(T)$ is empty.
\end{itemize}
Now if $\frac{x+y}{T} \not \in \Z$ and if the interval $\left( \frac{x}{T} , \frac{x+y}{T} \right]$ contains an integer, then $\left \lfloor \frac{x+y}{T} \right \rfloor > \frac{x}{T}$ so that
$$0 \leqslant \frac{x+y}{\lfloor (x+y)/T \rfloor} - T < \frac{x+y}{x/T} - T = \frac{Ty}{x}$$
achieving the proof.
\end{proof}

\begin{remark}
Applying Lemmas~\ref{le1} and~\ref{le3} to our problem yields the estimate
$$\sum_{x < n \leqslant x+y} F(n) = \frac{y a_s}{s+1} (\log x)^{s+1} + O_\varepsilon \left (  xy^{\kappa-1} \, \left( \log x \right)^\beta + y (\log x)^{\max \left (s,\delta  + m - 1 \right )} + x^\varepsilon \right)$$
as soon as $x^{1/2} \leqslant y \leqslant x$, uniformly for any function $F$ satisfying \eqref{e1} and \eqref{e2}. Thus, Theorem~\ref{t1} substantially improves on the error term, and the rest of the text is devoted to show how one can get such an improvement.
\end{remark}
 
\section{Estimates for Hooley-type $\Delta$-functions}
\label{s5}

\noindent
We easily derive from \cite[Theorem 70]{hal} and partial summation the following result.

\begin{lemma}
\label{le4} 
For any $x \geqslant 1$ sufficiently large and any fixed integer $r \geqslant 2$
$$\sum_{n \leqslant x} \frac{\Delta_r (n)}{n} \ll_{r,\varepsilon} (\log x)^{1 + \epsilon_r (x)}$$
where $\epsilon_r \left( x \right)$ is defined in \eqref{e3}.
\end{lemma}

\noindent
Lemma~\ref{le4} implies the following short interval result for Hooley's $\Delta_r$-functions.

\begin{corollary}
\label{cor1}
For any large real numbers $1 \leqslant y \leqslant x$, small real number $\varepsilon \in \left( 0, \frac{1}{2} \right]$ and any fixed integer $r \geqslant 2$
$$\sum_{x < n \leqslant x+y} \Delta _r (n) \ll_{r,\varepsilon} y (\log x)^{\epsilon_r (x)} + x^\varepsilon$$
where $\epsilon_r \left( x \right)$ is defined in \eqref{e3}.
\end{corollary}

\begin{proof}
If $1 \leqslant y \leqslant x^{\varepsilon/2}$, then
$$\sum_{x < n \leqslant x+y} \Delta _r (n) \ll y \, \max_{x < n \leqslant x+y} \Delta _r (n) \ll_{r,\varepsilon} y \, \max_{x < n \leqslant x+y} n^{\varepsilon / 2} \ll_{r,\varepsilon} x^\varepsilon$$
whereas the case $x^{\varepsilon/2} \leqslant y \leqslant x$ follows from \cite[Corollary~4]{nai} and Lemma~\ref{le4}.
\end{proof}

\noindent
Borrowing an idea from \cite{hal}, we get the following upper bound.

\begin{lemma}
\label{le5} 
For all $r,n,N \in \Z_{\geqslant 1}$ 
$$\sum_{\substack{ d \mid n \\ N < d \leqslant 2N}} \tau _r (d) \leqslant \left( \log 2eN \right)^{r-1} \Delta_{r+1} (n).$$
\end{lemma}

\begin{proof}
The result is obvious for $r=1$, so that we may suppose $r \geqslant 2$. For all $n \in \Z_{\geqslant 1}$ and $k_1, \ldots ,k_{r-1} \in \R$ , we set
$$\Delta \left( n;k_1,\ldots ,k_{r-1}\right) := \sum_{\substack{d_1 \cdots d_{r-1} \mid n \\ e^{k_i} < d_i \leqslant e^{k_i+1}}} 1$$
so that from \cite[page~122]{hal}
$$\tau _r (d) = \sum_{k_1=0}^{\left \lfloor \log d \right \rfloor} \cdots \sum_{k_{r-1}=0}^{\left \lfloor \log d \right \rfloor} \Delta \left( d;k_1,\ldots,k_{r-1}\right)$$
and hence
\begin{eqnarray*}
   \sum_{\substack{ d \mid n \\ N < d \leqslant 2N}} \tau _r (d) &=& \sum_{\substack{ d \mid n \\ N < d \leqslant 2N}} \sum_{k_1=0}^{\left \lfloor \log d \right \rfloor} \cdots \sum_{k_{r-1}=0}^{\left \lfloor \log d \right \rfloor} \Delta \left( d;k_1,\ldots,k_{r-1}\right) \\
   & \leqslant & \sum_{k_1=0}^{\left \lfloor \log 2N \right \rfloor} \cdots \sum_{k_{r-1}=0}^{\left \lfloor \log 2N \right \rfloor} \sum_{\substack{ d \mid n \\ N < d \leqslant 2N}} \,\sum_{\substack{d_1 \cdots d_{r-1} \mid d \\ e^{k_i} < d_i \leqslant e^{k_i+1}}} 1 \\
   & \leqslant & \sum_{k_1=0}^{\left \lfloor \log 2N \right \rfloor} \cdots \sum_{k_{r-1}=0}^{\left \lfloor \log 2N \right \rfloor} \sum_{\substack{d_1 \cdots d_{r-1} d_r \mid n \\ e^{k_i} < d_i \leqslant e^{k_i+1} \\ N / \left( d_1 \cdots d_{r-1}\right) < d_r \leqslant 2N/\left( d_1\cdots d_{r-1}\right)}} 1\\
   & \leqslant & \sum_{k_1=0}^{\left \lfloor \log 2N \right \rfloor} \cdots \sum_{k_{r-1}=0}^{\left \lfloor \log 2N \right \rfloor} \Delta _{r+1}\left( n\right) \leqslant \left( \log 2eN\right) ^{r-1}\Delta _{r+1}\left( n\right)
\end{eqnarray*}
as asserted.
\end{proof}

\noindent
Now from Corollary~\ref{cor1} and Lemma~\ref{le5}, we are in a position to establish the main result of this section.

\begin{proposition}
\label{pro1}
Let $z \geqslant 1$ be any real number, $4 \leqslant H \leqslant N \leqslant z$ and $m \geqslant 1$ be integers. Then
$$\sum_{N < n \leqslant 2N} \tau_m(n) \min \left( 1,\frac {1}{H \left\| z/n \right\| }\right) \ll NH^{-1} \log H (\log N)^{m-1} \left( \log z\right)^{\epsilon_{m+1} (z)} + z^{\varepsilon}$$
where $\epsilon_{m+1} (z)$ is defined in \eqref{e3}.
\end{proposition}

\begin{proof}
Define $K := \left \lfloor \log H/\log 2 \right \rfloor \geqslant 2$. Generalizing \cite[Lemma~2.2]{bor0}, we get
$$\sum_{N < n \leqslant 2N} \tau_m(n) \min \left( 1 , \frac{1}{H \| z/n \|} \right) \ll H^{-1} \sum_{N < n \leqslant 2N} \tau_m(n) + \sum_{k=0}^{K-2} 2^{-k} \sum_{\substack{ N < n \leqslant 2N \\ \| z/n \| < 2^k H^{-1}}} \tau_m(n)$$
and using Lemma~\ref{le5} and Corollary~\ref{cor1} we obtain for any $k \in \left\lbrace 0, \dotsc, K-2 \right\rbrace$
\begin{eqnarray*}
   \sum_{\substack{ N < n \leqslant 2N \\ \| z/n \| < 2^k H^{-1}}} \tau_m(n) & \leqslant & \sum_{N < n \leqslant 2N} \tau_m(n) \left( \left \lfloor \frac{z}{n} + 2^k H^{-1} \right \rfloor - \left \lfloor \frac{z}{n} - 2^k H^{-1} \right \rfloor \right) \\
   & \leqslant & \sum_{z-2^{k+1} NH^{-1} < \ell \leqslant z+2^{k+1} NH^{-1}} \sum_{\substack{n \mid \ell \\ N < n \leqslant 2N}} \tau_m(n) \\
   & \leqslant & (\log 2eN)^{m-1} \sum_{z-2^{k+1} NH^{-1} < \ell \leqslant z+2^{k+1} NH^{-1}} \Delta_{m+1} (\ell) \\
   & \ll & (\log N)^{m-1} \left\lbrace 2^k NH^{-1} (\log z)^{\epsilon_{m+1} (z)} + z^\varepsilon \right\rbrace 
\end{eqnarray*}
and thus
\begin{eqnarray*}
   & & \sum_{N < n \leqslant 2N} \tau_m(n) \min \left( 1 , \frac{1}{H \| z/n \|} \right) \\
   & \ll & NH^{-1} (\log N)^{m-1} + \sum_{k=0}^{K-2} 2^{-k} \left\lbrace 2^k NH^{-1} (\log z)^{\epsilon_{m+1} (z)} + z^\varepsilon \right\rbrace (\log N)^{m-1} \\
   & \ll & NH^{-1} \log H (\log N)^{m-1} (\log z)^{\epsilon_{m+1} (z)} + z^{2\varepsilon}
\end{eqnarray*}
achieving the proof.
\end{proof}

\section{Proof of Theorem~\ref{t1}}
\label{s6}

\noindent
Let $e^e < y \leqslant x$ and $\max \left (y,\frac{x}{y} \right) \leqslant T \leqslant x$. From Lemma~\ref{le3} with $g=\mathbf{1}$, \eqref{e2} and Shiu's theorem \cite{shi}, we get
\begin{eqnarray*}
   \sum_{x < n \leqslant x+y} F (n) &=& \sum_{d \leqslant T} f(d) \left( \left \lfloor \frac{x+y}{d} \right \rfloor - \left \lfloor \frac{x}{d} \right \rfloor \right) + \sum_{k \leqslant \frac{x}{T}} \sum_{\frac{x}{k} < d \leqslant \frac{x+y}{k}} f(d) \\
   & & \\
   & & {} + O_{A} \left( (\log T)^\delta \sum_{T < n \leqslant T+y} \tau_m(n) \right) \\
   & & \\
   &=& y \sum_{d \leqslant T} \frac{f(d)}{d} - \sum_{y < d \leqslant T} f(d) \left( \psi \left( \frac{x+y}{d} \right) - \psi \left( \frac{x}{d} \right) \right) + \sum_{k \leqslant \frac{x}{T}} \sum_{\frac{x}{k} < d \leqslant \frac{x+y}{k}} f(d) \\
   & & {} + O \left( \sum_{d \leqslant y} \left | f(d) \right |\right) + O_{A,\varepsilon} \left( y (\log T)^{\delta+m-1} + x^{ \varepsilon} \right) \\
   & & \\
   &:=& y S_1 - S_2 + S_3  + O_{A,\varepsilon} \left( y (\log T)^{\delta+m-1} + x^{ \varepsilon} \right).
\end{eqnarray*}
For $S_1$, using partial summation and \eqref{e1}, we get
   \begin{eqnarray}
      \sum_{d \leqslant T} \frac{f(d)}{d} &=& \sum_{j=0}^s a_j (\log T)^j + \sum_{j=0}^s a_j \int_1^T \frac{(\log t)^j}{t} \, \textrm{d}t + O \left( 1 \right) \notag \\
      & & \notag \\
      &=& \frac{a_s}{s+1} (\log T)^{s+1} + O \left( \log^s T \right). \label{e5}
   \end{eqnarray}
For $S_3$, using Lemma~\ref{le1}, we derive
\begin{eqnarray*}
   S_3 &=& \sum_{k \leqslant \frac{x}{T}} \left\lbrace \frac{y a_s}{k} \log^s \frac{x}{k} + O \left( \frac{y}{k} \log^{s-1} x  + \left( \frac{x}{k} \right)^\kappa \log^\beta x \right) \right\rbrace \\
   & & \\
   &=& y a_s \sum_{k \leqslant \frac{x}{T}} \frac{1}{k} \log^s \frac{x}{k} + O \left( y \log^s x + xT^{\kappa - 1} \log^\beta x \right)
\end{eqnarray*}
and using
$$\sum_{k \leqslant z} \tfrac{1}{k} \left( \log \tfrac{x}{k} \right)^j = \tfrac{1}{j+1} \left\lbrace (\log x)^{j+1} - \left (\log \tfrac{x}{z} \right)^{j+1} \right\rbrace + O_j \left( \log^j x \right) \quad \left( j \in \Z_{\geqslant 0}, \ 1 \leqslant z \leqslant x \right)$$
we get
\begin{equation}
   S_3 = \frac{y a_s}{s+1} \left\lbrace (\log x)^{s+1} - \left (\log T \right)^{s+1} \right\rbrace + O \left( y \log^s x + xT^{\kappa - 1} \log^\beta x \right). \label{e6}
\end{equation}
It remains to estimate $S_2$. To this end, we first split the interval $\left( y,T \right]$ into $O \left (\log \frac{T}{y} \right )$ dyadic sub-intervals of the shape $\left( N,2N \right]$ giving
$$| S_2 | \ll \max_{y < N \leqslant T} \left | \Sigma_F(N,x,y) \right | \log (T/y)$$
where $\Sigma_F(N,x,y)$ is given in \eqref{e4}. From Lemma~\ref{le2} and Abel summation, we get
\begin{eqnarray*}
   \left | \Sigma_F(N,x,y) \right | & \ll & \frac{y}{N} \max_{N \leqslant N_1 \leqslant 2N} \, \max_{x \leqslant z \leqslant x+y} \, \sum_{N < n \leqslant N_1} \left | f(n) \right| \left | \sum_{h \leqslant H}  e\left( \frac{hz}{n} \right) \right| \\
   & & {} + \max_{x \leqslant z \leqslant x+y} \, \sum_{N < n \leqslant 2N} \left | f(n) \right| \min \left( 1,\frac{1}{H \left\| z/n\right\| }\right)  \\
   & & \\
   & \ll_A & \left\lbrace \frac{yH}{N} \max_{x \leqslant z \leqslant x+y} \,  \sum_{N < n \leqslant 2N} \tau_m(n) \min \left( 1,\frac{1}{H \left\| z/n\right\| }\right) \right.   \\
   & & \left. {}  +  \max_{x \leqslant z \leqslant x+y} \, \sum_{N < n \leqslant 2N} \tau_m(n) \min \left( 1,\frac{1}{H \left\| z/n \right\| }\right) \right\rbrace (\log N)^\delta  \\
\end{eqnarray*}
and Proposition~\ref{pro1} implies that, for $H \geqslant 4$ and $\max(H,y,\frac{x}{y}) \leqslant N \leqslant T \leqslant x$, we have
\begin{eqnarray*}
   \left | \Sigma_F(N,x,y) \right | & \ll_{A,\varepsilon} & \log H (\log N)^{\delta+m-1} (\log x)^{\epsilon_{m+1} (x)} \left( y + N H^{-1} \right) + x^{2 \varepsilon} \left( yHN^{-1} + 1 \right) \\
   & \ll_{A,\varepsilon} & y \log (N/y) (\log N)^{\delta+m-1} (\log x)^{\epsilon_{m+1} (x)} + x^{2 \varepsilon}
\end{eqnarray*}
with the choice of $H = 4 \left \lfloor N y^{-1}\right \rfloor$. If $\max \left(y,\frac{x}{y} \right) \leqslant T \leqslant x$, adding the contributions from \eqref{e5} and \eqref{e6} we deduce that
\begin{eqnarray*}
   \sum_{x < n \leqslant x+y} F (n) &=& \frac{y a_s}{s+1} \, (\log x)^{s+1} \\
   & & {} + O \left \{ xT^{\kappa - 1} (\log x)^\beta + y \left( \log \tfrac{T}{y} \right)^2 (\log T)^{\delta+m-1} (\log x)^{\epsilon_{m+1} (x)} \right. \\
   & & \left. {} + y \log^s x + x^{2 \varepsilon} \right \}
\end{eqnarray*}
the term $y(\log T)^{\delta+m-1}$ being absorbed by $y \left( \log (T/y) \right)^2 (\log T)^{\delta+m-1} (\log x)^{\epsilon_{m+1} (x)}$. The asserted result then  follows from choosing $T = y \exp \left( (\log x)^{1/4} \right)$.
\qed

\section{Applications and examples}
\label{s7}

\noindent
In this section, the following additional arithmetic functions will appear. \\

\noindent
\begin{scriptsize} $\triangleright$ \end{scriptsize} For any $k \in \Z_{\geqslant 2}$, let $\mu_k$ be the characteristic function of the set of $k$-free numbers. Note that $\mu_2 = \mu^2$. \\

\noindent
\begin{scriptsize} $\triangleright$ \end{scriptsize} For any $n,k \in \Z_{\geqslant 2}$, $\tau_{(k)}(n)$ counts the number of $k$-free divisors of $n$, with the convention $\tau_{(k)}(1)=1$. Note that $\tau_{(2)} = 2^{\omega}$ where, as usual, $\omega(n)$ is the number of distinct prime factors of $n$. \\

\noindent
\begin{scriptsize} $\triangleright$ \end{scriptsize} For any $k \in \Z_{\geqslant 1}$, $\Lambda_k$ is the $k$th von Mangoldt's function defined by
$\Lambda_k = \mu \star \log^k$. Similarly, if $g$ is any arithmetic function satisfying $g(1) \neq 0$, the von Mangoldt function $\Lambda_g$ attached to $g$ is implicitely defined by the equation $\Lambda_g \star g = g \times \log$.

\subsection{Example~1}

\noindent
In \cite{gar}, it is shown that, uniformly for $x^{1/2} \log x < y < x^{1/2} (\log x)^{5/2}$
$$\sum_{x < n \leqslant x+y} \tau_4 (n) = \tfrac{1}{6} \, y (\log x)^3 + O \left( (xy)^{1/3} (\log x)^{2/3} \right).$$

\noindent
Let $k \in \Z_{\geqslant 2}$ and take $f=\tau_{k-1}$. Applying Theorem~\ref{t1} with $m=k-1$, $s=k-2$, $\delta = 0$, $\kappa = 1 - \frac{2}{k}$, we get

\begin{corollary}
\label{cor2}
For any $x^{1/2} e^{-\frac{1}{2}(\log x)^{1/4}} \leqslant y \leqslant xe^{-(\log x)^{1/4}}$ and any $\varepsilon \in \left] 0,\frac{1}{2} \right]$ 
\begin{eqnarray*}
   \sum_{x < n \leqslant x+y} \tau_k (n) &=& \frac{y (\log x)^{k-1}}{(k-1)!} \\
   & & {} + O_\varepsilon \left( x y^{- 2/k+\varepsilon}e^{ \left( - \frac{2}{k} +\varepsilon \right) (\log x)^{1/4}} + y (\log x)^{k-3/2 + \epsilon_k(x)} + x^\varepsilon \right)
\end{eqnarray*}
the term $y^\varepsilon$ being omitted when $k=2$. In particular,
$$\sum_{x < n \leqslant x+y} \tau_4 (n) = \tfrac{1}{6} \, y (\log x)^3 + O \left( y (\log x)^{5/2+\epsilon_4(x)}  \right)$$
for all $x^{2/3+\varepsilon} e^{- \frac{1}{3} (\log x)^{1/4}} \leqslant y \leqslant x e^{-(\log x)^{1/4}}$.
\end{corollary}

\subsection{Example~2}

\noindent
In this example, $F(n)$ is either $\tau(n)^2$, or $\tau \left( n^3 \right)$. Improving on a result from \cite{gar2}, Zhai \cite[Corollary~4]{zhai} showed that
$$\sum_{x < n \leqslant x+y} F(n) \sim C_F \, y (\log x)^3$$
for $y=o(x)$ with 
$$\frac{y}{x^{1/2} \log x} \to \infty \quad \textrm{lorsque} \quad x \to \infty$$
and where $C_F = \dfrac{1}{6 \zeta(2)}$ if $F = \tau^2$, and\footnote{There is a misprint on $C_F$ in \cite[Corollary~1]{gar2} for which the coefficient $\frac{1}{6}$ has been forgotten in both cases.} 
$$C_F = \frac{1}{6} \prod_p \left( 1- \frac{1}{p} \right)^2 \left( 1 + \frac{2}{p} \right) \quad \textrm{if\ } F = \tau \circ \id^3.$$
Let $f$ be the Eratosthenes transform of $F$. Note that, if $F = \tau^2$, then $f = \tau \circ \id^2$, and if $F = \tau \circ \id^3$, then $f=3^{\omega}$, so that in both cases
$$\left | f(n) \right | \leqslant \tau_3(n) \quad \left( n \in \Z_{\geqslant 1} \right)$$
and by \cite[Main Theorem]{zhai}, we have in both cases
$$\sum_{n \leqslant x} f(n) = 3C_F \,x (\log x)^2 + A x \log x + Bx  + O \left( x^{1/2} (\log x)^4 \right)$$
where $A,B \in \R$. We may apply Theorem~\ref{t1} with $\delta = 0$, $A=1$, $s=2$, $m=3$, $\beta = 4$ and $\kappa = \frac{1}{2}$, giving the following more precise version of Zhai's result.

\begin{corollary}
\label{cor3}
The function $F$ being indifferently either $\tau^2$ or $\tau \circ \id^3$, for any $x^{1/2} e^{-\frac{1}{2}(\log x)^{1/4}} \leqslant y \leqslant xe^{-(\log x)^{1/4}}$ and any $\varepsilon \in \left] 0,\frac{1}{2} \right]$
\begin{eqnarray*}
   \sum_{x < n \leqslant x+y} F (n) &=& C_F \, y (\log x)^3 \\
   & & {} + O_\varepsilon \left( x y^{-1/2} \, e^{-\frac{1}{2}(\log x)^{1/4}}(\log x)^4 + y (\log x)^{5/2+\epsilon_4(x)}  + x^\varepsilon \right)
\end{eqnarray*}
where $C_F = \dfrac{1}{6 \zeta(2)}$ if $F = \tau^2$, and
$$C_F = \frac{1}{6} \prod_p \left( 1- \frac{1}{p} \right)^2 \left( 1 + \frac{2}{p} \right)$$
if $F = \tau \circ \id^3$. In particular, if $x^{2/3} e^{-\frac{1}{3}(\log x)^{1/4}}\log x \leqslant y \leqslant xe^{-(\log x)^{1/4}}$
$$\sum_{x < n \leqslant x+y} F (n) = C_F \, y (\log x)^3 + O_\varepsilon \left( y (\log x)^{5/2+\epsilon_4(x)} \right).$$
\end{corollary}

\subsection{Example~3}

\noindent
Let $k \in \Z_{\geqslant 2}$ and take $f=\mu_k$. Theorem~\ref{t1} applied with the values $\delta = s = \beta = 0$, $m=1$ and $\kappa = \frac{1}{k}$ gives the following corollary.

\begin{corollary}
\label{cor4}
For $x^{1/2} e^{-\frac{1}{2}(\log x)^{1/4}} \leqslant y \leqslant xe^{-(\log x)^{1/4}}$ and any $\varepsilon \in \left( 0,\frac{1}{2} \right]$
$$\sum_{x < n \leqslant x+y} \tau_{(k)} (n) = \frac{y \log x}{\zeta(k)} + O_\varepsilon \left( x y^{-1+1/k} \, e^{(-1+1/k)(\log x)^{1/4}} + y (\log x)^{1/2+\epsilon_2(x)}  + x^\varepsilon \right).$$
In particular, if $x^{\frac{k}{2k-1}} \, e^{\frac{1-k}{2k-1} \, (\log x)^{1/4}} \leqslant y \leqslant x  e^{-(\log x)^{1/4}}$, then
$$\sum_{x < n \leqslant x+y} \tau_{(k)} (n) = \frac{y \log x}{\zeta(k)} + O \left( y (\log x)^{1/2+\epsilon_2(x)}  \right).$$
For instance, when $k=2$, we get
$$\sum_{x < n \leqslant x+y} 2^{\omega(n)} = \frac{y \log x}{\zeta(2)} + O \left( y (\log x)^{1/2+\epsilon_2(x)}  \right)$$
for all $x^{2/3} \, e^{- \frac{1}{3} \, (\log x)^{1/4}} \leqslant y \leqslant x e^{-(\log x)^{1/4}}$.
\end{corollary}

\subsection{Example~4}

\noindent
Let $k \in \Z_{\geqslant 2}$ and take $f=\tau_{(k)}$. Theorem~\ref{t1} with $s=1$, $m=2$, $\delta = 0$ and $\kappa$ given below gives

\begin{corollary}
\label{cor5}
For all $x^{1/2} e^{-\frac{1}{2}(\log x)^{1/4}} \leqslant y \leqslant xe^{-(\log x)^{1/4}}$ and any $\varepsilon \in \left( 0,\frac{1}{2} \right]$ 
$$\sum_{x < n \leqslant x+y} \left( \tau \star \mu_k \right) (n) = \frac{y (\log x)^2}{2 \zeta(k)} + O_\varepsilon \left( x y^{\kappa - 1} e^{(\kappa-1)(\log x)^{1/4}} + y (\log x)^{3/2+\epsilon_3(x)} + x^\varepsilon \right)$$
where
$$\kappa := \begin{cases} 1/k, & \textrm{si\ } k \in \{2,3\} \\ \frac{131}{416} + \varepsilon, & \textrm{si\ } k \geqslant 4. \end{cases}$$
\end{corollary}

\subsection{Example~5}

\noindent
We use Theorem~\ref{t1} with $f=\mu^2 \times 2^\omega$. In \cite{gor}, it is shown that, for any real number $z \geqslant 1$ sufficiently large\footnote{In fact, the authors in \cite{gor} give an error term of the shape $O_\varepsilon \left(z^{1/2 + \varepsilon} \right)$, but a close inspection of their proof reveals that it can be sharpened to $\ll z^{1/2} (\log z)^6$.}
$$\sum_{n \leqslant z} \mu^2(n) 2^{\omega(n)} = A z \log z + B z + O \left( z^{1/2} (\log z)^6 \right)$$
where $A=\prod_p \left( 1 - \tfrac{1}{p} \right)^2 \left( 1 + \tfrac{2}{p} \right)$ et $B = A \left( 2 \gamma - 1 + 6 \sum_p \frac{(p-1) \log p}{p^2(p+2)} \right)$. Thus, we can apply Theorem~\ref{t1} with $s=1$, $m=2$, $\delta = 0$, $\beta = 6$ and $\kappa= \frac{1}{2}$, giving

\begin{corollary}
\label{cor6}
For any $\varepsilon \in \left( 0,\frac{1}{2} \right]$ and all $x^{1/2} e^{-\frac{1}{2}(\log x)^{1/4}} \leqslant y \leqslant x e^{-(\log x)^{1/4}}$
\begin{eqnarray*}
   \sum_{x < n \leqslant x+y} 3^{\omega(n)} &=& \tfrac{1}{2} \prod_p \left( 1 - \tfrac{1}{p} \right)^2 \left( 1 + \tfrac{2}{p} \right) \, y (\log x)^2 \\
   & & {} + O \left ( x y^{-1/2} e^{- \frac{1}{2} (\log x)^{1/4}} (\log x)^6 + y (\log x)^{3/2+\epsilon_3(x)} + x^\varepsilon \right ).
\end{eqnarray*}
\end{corollary}

\subsection{Example~6}

\noindent
The function $F$ does not need to be multiplicative. Suppose that we have at our disposal an arithmetic function $g$ such that $|g(n)| \leqslant A$, $g(1) \neq 0$ and
$$\sum_{n \leqslant z} g(n) = a z + O \left ( z^\kappa (\log z)^\beta \right) \quad \left( z \geqslant 1 \right)$$
with $a \in \C \setminus \{0\}$, $\kappa \in \left[ 0,1 \right)$, $\beta \geqslant 0$. By Abel summation, we get
$$\sum_{n \leqslant z} g(n) \log n = az(\log z - 1) + O \left( z^\kappa (\log z)^{\beta+1} \right)$$
so that Theorem~\ref{t1} may be used with $f=g \times \log$ and $m=s=\delta=1$, giving the following estimate.

\begin{corollary}
\label{cor7}
For any $\varepsilon \in \left( 0,\frac{1}{2} \right]$ and all $x^{1/2} e^{-\frac{1}{2}(\log x)^{1/4}} \leqslant y \leqslant x e^{-(\log x)^{1/4}}$
\begin{eqnarray*}
   \sum_{x < n \leqslant x+y} \Lambda_g (n) \, G(n) &=& \tfrac{1}{2} a y (\log x)^2 + O \Bigl ( xy^{\kappa-1} e^{(\kappa-1) (\log x)^{1/4}} (\log x)^{\beta+1} \\
   & & {} + y (\log x)^{3/2+\epsilon_2(x)} + x^\varepsilon \Bigr )
\end{eqnarray*}
where $G := g \star \mathbf{1}$.
\end{corollary}

\noindent
For instance, when $g = \mu^2$, so that $a=\zeta(2)^{-1}$, $\kappa= \frac{1}{2}$ and $\beta = 0$, we get for all $x^{1/2} e^{-\frac{1}{2}(\log x)^{1/4}} \leqslant y \leqslant x e^{-(\log x)^{1/4}}$ and any $\varepsilon \in \left( 0,\frac{1}{2} \right]$
\begin{eqnarray*}
   \sum_{x < n \leqslant x+y} \Lambda_{\mu^2} (n) \, 2^{\omega(n)} &=& \frac{y (\log x)^2}{2 \zeta(2)} + O \Bigl ( xy^{-1/2} e^{- \frac{1}{2} (\log x)^{1/4}} \log x  \\
   & & {} + y (\log x)^{3/2+\epsilon_2(x)} + x^\varepsilon \Bigr ).
\end{eqnarray*}
In particular, if $x^{2/3} \, e^{- \frac{1}{3} \, (\log x)^{1/4}} \leqslant y \leqslant x e^{-(\log x)^{1/4}}$, then
$$\sum_{x < n \leqslant x+y} \Lambda_{\mu^2} (n) \, 2^{\omega(n)} = \frac{y (\log x)^2}{2 \zeta(2)} + O \left( y (\log x)^{3/2+\epsilon_2(x)} \right).$$

\subsection{Example~7}

\noindent
Let $k \in \Z_{\geqslant 1}$ and take $f=\log^k \star \log^k$. For any $n \in \Z_{\geqslant 1}$, we have $0 \leqslant f(n) \leqslant 4^{-k} \tau(n) (\log n)^{2k}$. Furthermore, it is known \cite{kui} that for any $\varepsilon \in \left( 0,\frac{1}{2} \right]$ and any real number $z \geqslant 1$
$$\sum_{n \leqslant z} f(n) = zP_{2k+1}(\log z) + O_\varepsilon \left( z^{1/3+\varepsilon} \right)$$
where $P_{2k+1}$ is a polynomial of degree $2k+1$ and leading coefficient $\frac{(k!)^2}{(2k+1)!}$, so that Theorem~\ref{t1} can be used with $s=2k+1$, $m=2$, $\delta = 2k$ and $\kappa = \frac{1}{3}$, giving

\begin{corollary}
\label{cor8}
For all $x^{1/2} e^{-\frac{1}{2}(\log x)^{1/4}} \leqslant y \leqslant x e^{-(\log x)^{1/4}}$, any $k \in \Z_{\geqslant 1}$ and any $\varepsilon \in \left( 0,\frac{1}{2} \right]$ 
\begin{eqnarray*}
   & & \sum_{x < n \leqslant x+y} \left( \Lambda_k \star \tau \star \log^k \right)(n) = \frac{(k!)^2 \, y (\log x)^{2k+2}}{(2k+2)!} \\
   & & {} + O_{k,\varepsilon} \biggl ( xy^{-2/3+\varepsilon}e^{\left( -\frac{2}{3} + \varepsilon \right)(\log x)^{\frac{1}{4}}} + y (\log x)^{2k+\frac{3}{2}+\epsilon_3(x)} + x^\varepsilon \biggr ).
\end{eqnarray*}
\end{corollary}

\end{document}